\documentclass[preprint,12pt]{elsarticle}

\usepackage[english]{babel}
\usepackage{graphicx,epstopdf,epsfig}
\usepackage{amsfonts,epsfig,fancyhdr,graphics, hyperref,amsmath,amssymb}
\usepackage{amsthm}
\newtheorem{theorem}{Theorem}[section]
\newtheorem{lemma}[theorem]{Lemma}
\newtheorem{corollary}[theorem]{Corollary}
\newtheorem{proposition}[theorem]{Proposition}
\newtheorem{definition}[theorem]{Definition}

\usepackage{xcolor}

\journal{Linear Algebra and its Applications}

\begin{document}

\begin{frontmatter}



\title{A Family of Completely PPT Maps}


\author[label1]{Kennett L. Dela Rosa}
\author[label2]{Iana Angela C. Fajardo}
\author[label1]{Cedie B. Reyes}

\affiliation[label1]{organization={Institute of Mathematics, University of the Philippines Diliman},
            city={Quezon City},
            postcode={1101}, 
            state={NCR},
            country={Philippines}}

\affiliation[label2]{organization={Graduate School of Science and Technology, Sophia University},
addressline={7-1 Kioi-cho, Chiyoda-ku, Tokyo},
postcode={102-8554},
country={Japan}}

\begin{abstract}
Given $\alpha\in \mathbb R$ and $ B_1,\ldots,B_q\in M_{m,n}$, consider the map
\[\Phi^{\pm}_{\alpha,\mathfrak{B}}(X)=\alpha \textup{tr}(X) I_m\pm \sum_{s=1}^qB_sXB_s^*\ \textup{for all}\ X\in M_n.\] In this study, necessary and sufficient conditions on $\alpha $ and $B_1,\ldots,B_q$ are given that guarantee $\Phi_{\alpha,\mathfrak{B}}^\pm$ is $k$-positive, $k$-copositive, and $k$-PPT, respectively.
As application, known results about the map $\alpha \mbox{tr}(X)I_n\pm X$ are recovered. Furthermore, singular value inequalities associated with these positivity constraints are considered.

\end{abstract}



\end{frontmatter}


\section{Introduction}
\label{}
Completely positive maps were introduced in \cite{stinespring} in connection with operator dilations. Considerable attention has also been devoted to finding positive maps that are not completely positive, as such maps can be used to detect entangled states. On the other hand, there is also interest in studying completely copositive maps as they are related to the notion of decomposable maps. For works related to these ideas, see for instance  \cite{alberti,chefles,cho,chru1,chru2,chru3,chru4,eom,horodecki,hou2,huang,labuschagne,qi,sengupta1,sengupta2}.

The map $\Phi(X)=(m-1)\textup{tr}(X)I_m-X$ is the classical example of an $(m-1)$-positive map that is not $m$-positive \cite{choi}. In 2014, Lin showed that the map $\Phi^+(X)=\textup{tr}(X)I_m+X$ is completely PPT \cite{lin2014}. Lin noted that $\Phi^+$ is known to be completely positive, and the main effort was to show its complete copositivity. To do this, Lin used Schur complement techniques and matrix inequalities. In 2016, Lin considered a different map $\Phi^-(X)=\textup{tr}(X)I_m-X$ \cite{lin2016}. He used the Choi matrix to prove that $\Phi^-$ is completely copositive. In addition, Lin remarked that the copositivity of $\Phi^-$ may be proved using techniques from \cite{lin2014} (see also \cite{yietal} for some applications of $\Phi^+$ and $\Phi^-$). In 2019, Zhang showed that $\Phi(X)=\min\{k,m\}\textup{tr}(X)I_m-X$ is $k$-PPT \cite{pzhang}. In 2020, Fu, Lau, and Tam gave conditions equivalent to $\Phi^{\pm}_\alpha(X)=\alpha\textup{tr}(X)I_m\pm X$ being completely PPT  \cite{fulautam}. Moreover, singular value inequalities associated to $\Phi^\pm_{\alpha}$ were also shown. 

Motivated by Lin's remark in \cite{lin2014} about the rarity of PPT maps, we are interested in generalizing the results in the aforementioned papers for a map of the form
\[\Phi^{\pm}_{\alpha,\mathfrak{B}}(X)=\alpha \textup{tr}(X) I_m\pm \sum_{s=1}^qB_sXB_s^*\ \textup{for all}\ X\in M_n,\]
where $\alpha\in \mathbb R$ and $ B_1,\ldots,B_q\in M_{m,n}$. In Section \ref{main}, conditions are given equivalent to $\Phi_{\alpha,\mathfrak{B}}^\pm$ being $k$-positive, $k$-copositive, and $k$-PPT, respectively. In Section \ref{svi}, singular value inequalities associated to $\Phi_{\alpha,\mathfrak{B}}^\pm$ are explored.

\section{Preliminaries}

Given complex Hilbert spaces ${\cal H}$ and ${\cal K}$, let $B({\cal H}, {\cal K})$ denote the set of all bounded linear operators from ${\cal H}$ to ${\cal K}$. Set $B({\cal H}):=B({\cal H},{\cal H})$. If ${\cal H} $ has dimension $n$, it is identified with the set $\mathbb{C}^n$ of complex $n$-vectors. The set $B({\cal H}, {\cal K})$ is identified with the set $M_{m,n}$ of $m$-by-$n$ complex matrices when ${\cal H}$ and ${\cal K}$ have dimensions $n$ and $m$, respectively; set $M_n:=M_{n,n}$. The set of all positive semidefinite operators in $B({\cal H})$ is denoted by $B({\cal H})^+$.


Throughout the paper, $E_{ij}\in M_k$ denotes the standard basis matrix for $M_k$ (we write $E_{ij}^{(k)}$ to emphasize that the matrix is in $M_k$), i.e., the matrix in $M_k$ whose $(i,j)$-entry is $1$ and $0$ otherwise. Any $A\in M_{k}(B({\cal H}))$ can be written as $A=\sum_{i,j=1}^kE_{ij}\otimes A_{ij}=[A_{ij}]_{i,j=1}^k$ where $A_{ij}\in B({\cal H})$. Due to the underlying tensor structure, one natural way to define a \textit{partial transpose} of $A$ is $A^\tau=\displaystyle\sum_{i,j=1}^k E_{ij}^\top \otimes A_{ij}=[A_{ji}]_{i,j=1}^k$. If ${\cal H}$ has dimension $n$, then $A_{ij}\in M_n$ and another \textit{partial transpose} of $A$ is given by $A^\Gamma=\displaystyle\sum_{i,j=1}^k E_{ij}\otimes A_{ij}^\top=[A_{ij}^\top]_{i,j=1}^k$. When $k=1$, $A^\tau=A$ and $A^\Gamma=A^\top$. Observe that $A^\Gamma=(A^\tau)^\top$, and hence $A^\tau\in M_k(M_n)^+ $ if and only if $A^\Gamma\in M_k(M_n)^+$. For a general $X\in M_{m,n}$, let $\textup{vec}(X)=\begin{bmatrix}x_1\\ \vdots \\ x_n\end{bmatrix}$ where $x_1,\ldots,x_n\in \mathbb{C}^m$ are the columns of $X$; this defines a linear isometry $\textup{vec}:M_{m,n}\to \mathbb{C}^{mn}$ where $M_{m,n}$ is equipped with the Frobenius norm $\|X\|=\sqrt{\textup{tr}(X^*X)}$ while the codomain is equipped with the Euclidean norm $\|x\|=\sqrt{x^*x}$. A useful formula about the $\textup{vec}$ map is given by $\textup{vec}(AXB)=(B^\top \otimes A)\textup{vec}(X)$ \cite[Lemma 4.3.1]{HJ2}.

Detecting entanglement in quantum information theory motivates consideration of a special class of operators. The PPT Criterion or the Peres-Horodecki Criterion uses PPT matrices to detect entanglement.


\begin{definition}
An operator $A\in M_k(B({\cal H}))$ is said to be \textit{positive partial transpose (PPT)} if $A\in M_k(B({\cal H}))^+$ and $A^\tau\in M_k(B({\cal H}))^+$.
\end{definition}

A linear map $\Phi: B({\cal H})\to B({\cal K})$ is \textit{positive} if $\Phi(A)\in B({\cal K})^+$ whenever $A\in B({\cal H})^+$. Given $k\in \mathbb{N}$, define $I_k\otimes \Phi: M_k(B({\cal H}))\to M_k(B({\cal K}))$  by $(I_k\otimes \Phi)([A_{ij}]_{i,j=1}^k)=[\Phi(A_{ij})]_{i,j=1}^k$ where $[A_{ij}]_{i,j=1}^k\in M_k(B({\cal H}))$. We say that $\Phi$ is $k$\textit{-positive} if $I_k\otimes \Phi$ is positive. If $\Phi$ is $k$-positive for all $k\in \mathbb{N}$, then $\Phi$ is said to be \textit{completely positive}. 



Let $\mathbf{T}: M_n\to M_n$ denote the transpose operator on $M_n$, i.e., $\mathbf{T}(A)=A^\top$ for all $A\in M_n$. Following \cite{fulautam}, we consider the following maps.

\begin{definition}
Let $\Phi: M_n\to B({\cal K})$ be linear.
\begin{enumerate} [(i)]
\item Given  $k\in \mathbb{N}$, $\Phi$ is said to be $k$\textit{-copositive} if $\Phi\circ \mathbf{T}$ is $k$-positive.
\item The map $\Phi$ is said to be \textit{completely copositive} if $\Phi\circ \mathbf{T}$ is completely positive. 
\end{enumerate}
\end{definition}

\begin{definition}
Let $\Phi: M_n\to B({\cal K})$ be linear.
\begin{enumerate}[(i)] 
\item Given  $k\in \mathbb{N}$, $\Phi$ is said to be $k$\textit{-PPT} if $(I_k\otimes \Phi)(A)$ is PPT whenever $A\in  M_k(B({\cal H}))^+$.
\item The map $\Phi$ is said to be \textit{completely PPT} if $\Phi$ is $k$-PPT for all $k\in \mathbb{N}$. 
\end{enumerate}
\end{definition}

For simplicity, $1$-copositive and $1$-PPT maps are referred as copositive and PPT maps, respectively.
The proof of the next result closely follows that of \cite[Lemma 1]{fulautam}.

\begin{proposition} \label{equiv}
Let $\Phi: M_n\to B({\cal K})$ be linear. Given $k\in \mathbb{N}$, $\Phi$ is $k$-PPT if and only if $\Phi$ is $k$-positive and $k$-copositive. Consequently, $\Phi$ is completely PPT if and only if $\Phi$ is completely positive and completely copositive.
\end{proposition}

Suppose $\Phi: M_n\to B({\cal K})$ is linear. Its \textit{Choi matrix} $C(\Phi)$ is the operator matrix given by $C(\Phi)=[\Phi(E_{ij})]_{i,j=1}^n$. The Choi matrix can be used to characterize $k$-positivity of $\Phi$ (see \cite[Proposition 2.2]{hou1}). Now, observe that $C(\Phi\circ\mathbf{T})=[(\Phi\circ \mathbf{T})(E_{ij})]_{i,j=1}^n=[\Phi(E_{ji})]_{i,j=1}^n=C(\Phi)^\tau$. Since $k$-copositivity is defined in terms of the $k$-positivity of $\Phi\circ \mathbf{T}$, the next result follows from \cite[Proposition 2.2]{hou1}.

\begin{proposition} \label{k-cop}
Let $\Phi: M_n\to B({\cal K})$ be linear and $1\leq k\leq n$. The following are equivalent:
\begin{enumerate}[(i)]
\item $\Phi$ is $k$-copositive;
\item $\langle C(\Phi)^\tau x,x\rangle \geq 0$ for all $x=\displaystyle\sum_{p=1}^k y_p\otimes z_p$ with $y_p\otimes z_p\in \mathbb{C}^n\otimes {\cal K}$;
\item $(I_n\otimes P) C(\Phi)^\tau (I_n\otimes P)\in M_n(B({\cal K}))^+ $ for all rank $k$ orthogonal projection $P\in B({\cal K})$.
\end{enumerate}

\end{proposition}

\section{Elementary operators}
Consider $\Psi:M_n\to M_m$ defined by 
\begin{equation}\label{elemop}
\Psi(X)=\sum_{r=1}^pA_rXA_r^*-\sum_{s=1}^qB_sXB_s^*
\end{equation}
where $A_1,\ldots, A_p, B_1,\ldots,B_q\in M_{m,n}$. The map $\Psi$ is a special type of \textit{elementary operator} that sends self-adjoint matrices into itself \cite{hou0}. 

\begin{lemma}\label{KoPhi}
Let $\Psi:M_n\to M_m$ have the form \eqref{elemop} with $A_1,\ldots,A_p,B_1,\ldots,$ $B_q\in M_{m,n}$. Suppose $\Phi:M_n\to M_n$ is linear whose matrix representation with respect to $\{E_{11},\ldots,E_{n1},\ldots, E_{1n},\ldots, E_{nn}\}$ is given by $K(\Phi)=[K_{ij}]_{i,j=1}^n\in M_n(M_n)$. Then \[(\Psi\circ \Phi)(X)=\displaystyle\sum_{r=1}^p\displaystyle\sum_{i,j=1}^n A_rK_{ij}XE_{ij}^\top A_r^*-\displaystyle\sum_{s=1}^q\displaystyle\sum_{i,j=1}^nB_sK_{ij}XE_{ij}^\top B_s^*.\]
\end{lemma}
\begin{proof}
Consider
\begin{align*}
\textup{vec}&(\Psi\circ \Phi)(X)\\
&=\displaystyle\sum_{r=1}^p\textup{vec}(A_r\Phi(X)A_r^*)-\displaystyle\sum_{s=1}^q \textup{vec}(B_s\Phi(X)B_s^*)\\
&=\displaystyle\sum_{r=1}^p(\overline{A_r}\otimes A_r)\textup{vec}(\Phi(X))-\displaystyle\sum_{s=1}^q (\overline{B_s}\otimes B_s)\textup{vec}(\Phi(X))\\
&=\displaystyle\sum_{r=1}^p(\overline{A_r}\otimes A_r)K(\Phi)\textup{vec}(X)-\displaystyle\sum_{s=1}^q (\overline{B_s}\otimes B_s)K(\Phi)\textup{vec}(X).
\end{align*}

Note that
\begin{align*}
(\overline{A_r}\otimes A_r)K(\Phi)\textup{vec}(X)&=\displaystyle\sum_{i,j=1}^n\overline{A_r}E_{ij}\otimes A_rK_{ij}\textup{vec}(X)\\
&=\displaystyle\sum_{i,j=1}^n\textup{vec}(A_rK_{ij}XE_{ij}^\top A_r^*).
\end{align*}

Similarly,
\[(\overline{B_s}\otimes B_s)K(\Phi)\textup{vec}(X)=\displaystyle\sum_{i,j=1}^n\textup{vec}(B_sK_{ij}XE_{ij}^\top B_s^*).\] 
The claim follows from the injectivity of the linear map $\textup{vec}.$
\end{proof}

A straightforward application of Lemma \ref{KoPhi} yields the following result.

\begin{theorem}
Let $\Psi:M_n\to M_m$ have the form \eqref{elemop} with $A_1,\ldots,A_p,$ $B_1,\ldots, B_q\in M_{m,n}$. The following are equivalent:
\begin{enumerate}[(i)]
\item $\Psi$ is $k$-copositive;
\item $\displaystyle\sum_{r=1}^p (I_k\otimes A_r)P^\Gamma (I_k\otimes A_r^*)-\displaystyle\sum_{s=1}^q (I_k\otimes B_s)P^\Gamma (I_k\otimes B_s^*)\in M_k(M_m)^+$ for all rank $1$ orthogonal projection $P\in M_k(M_n)$;
\item For all (orthonormal) subset $\{x_1,\ldots,x_k\}\subset \mathbb{C}^n$,
\[\displaystyle\sum_{r=1}^p\sum_{i,j=1}^kE_{ij}\otimes A_r(x_ix_j^*)^\top A_r^*-\displaystyle\sum_{s=1}^q\sum_{i,j=1}^kE_{ij}\otimes B_s(x_ix_j^*)^\top B_s^*\in M_k(M_m)^+.\]

\end{enumerate}
\end{theorem}

\section{$k$-positivity and $k$-copositivity}\label{main}

The singular values of $B\in M_{m,n}$ are denoted as $\sigma_1(B)\geq \sigma_2(B)\geq \cdots \geq \sigma_{\min\{m,n\}}(B)\geq 0$. For a Hermitian matrix $H \in M_m,$ denote its eigenvalues as $\lambda_{max}(H) = \lambda_1(H) \geq \lambda_2(H) \geq \cdots \geq \lambda_m(H) = \lambda_{\min}(H)$.

\begin{definition}\label{phi_pm}
Let $\alpha\in \mathbb R$ and $\mathfrak{B}=\begin{bmatrix} B_1& \cdots& B_q\end{bmatrix}$ where $ B_1,\ldots,B_q\in M_{m,n}$. Consider $\Phi^{\pm}_{\alpha,\mathfrak{B}}: M_n\to M_m$ defined by 
\[\Phi^{\pm}_{\alpha,\mathfrak{B}}(X)=\alpha \textup{tr}(X) I_m\pm \sum_{s=1}^qB_sXB_s^*\ \textup{for all}\ X\in M_n.\]
Associated to this map are the matrix $Z=[\textup{vec}(B_1)\ \cdots\ \textup{vec}(B_q)]$ and the function $F(v)=\displaystyle\sum_{s=1}^q v_s B_s$ where $v_1,\ldots, v_q\in \mathbb C$. 
When $q=1$, write $\Phi^{\pm}_{\alpha,B}(X)=\alpha\textup{tr}(X)I_m\pm BXB^*$ for all $X\in M_n$; set $Z=\textup{vec}(B)$ and $F(v)=vB$ where $v\in \mathbb{C}$.
\end{definition}

Observe that the Choi matrix of $\Phi^{\pm}_{\alpha,\mathfrak{B}}$ is given by $\alpha I_{mn}\pm ZZ^*$ and $\textup{vec}(F(v))=Zv$ for all $v\in \mathbb{C}^q.$  

\begin{theorem}\label{beta_pm}
Let $\Phi^{\pm}_{\alpha,\mathfrak{B}}$ be as in Definition \ref{phi_pm}. Define \[\beta_k^\pm=\displaystyle\max_{\substack{V\in M_{m,k}\\ V^*V=I_k}}\{\lambda_{max}((I_n\otimes V)^*(\mp ZZ^*)(I_n\otimes V))\}.\] Then $\Phi^{\pm}_{\alpha,\mathfrak{B}}$ is $k$-positive if and only if $ \alpha\geq \beta_k^\pm.$
\end{theorem}
\begin{proof}
By \cite[Proposition 2.2]{hou1}, $\Phi^{\pm}_{\alpha,\mathfrak{B}}$ is $k$-positive if and only if $(I_n\otimes P)$ $(\alpha I_{mn}\pm ZZ^*)(I_n\otimes P)\in M_{mn}^+$ for all rank $k$ orthogonal projection $P\in M_m. $ By parametrizing $P=VV^*$ where $V\in M_{m,k}$ with $V^*V=I_k$ and upon simplification, we obtain the equivalent condition.
\end{proof}

Note that $\beta_k^+\leq 0 $ and $\beta_k^-\geq 0$. We give further characterization of $\beta_k^\pm$ in terms of $F(v)$.

\begin{theorem}\label{char_beta_pm}
Let $\beta_k^\pm$ be as in Theorem \ref{beta_pm}. Then \[\beta_k^+=-\min_{\substack{W\in M_{m,n}\\ \textup{rank}(W)\leq k\\ \textup{tr}(W^*W)=1}}\max_{\substack{v\in \mathbb C^q\\ \|v\|=1}}|\textup{tr}(F(v)^*W)|^2\] and 
\[\beta_k^-=\max_{\substack{v\in \mathbb C^q\\ \|v\|=1}}\sum_{j=1}^k\sigma_j(F(v)F(v)^*).\]
\end{theorem}
\begin{proof}
Let $V\in M_{m,k}$ with $V^*V=I_k.$ 

Consider the expression for $\beta_k^+.$ Then 
\[\begin{array}{rcl}\lambda_{max}((I_n\otimes V)^*(-ZZ^*)(I_n\otimes V))&=&-\lambda_{min}((I_n\otimes V)^*ZZ^*(I_n\otimes V))\\
&=&-\displaystyle\min_{\substack{x\in \mathbb C^{nk}\\ \|x\|=1}}x^*(I_n\otimes V)^*ZZ^*(I_n\otimes V)x.
\end{array}\]
For $x\in \mathbb{C}^{nk}$ with $\|x\|=1$, let $X\in M_{k,n}$ such that $\textup{vec}(X)=x.$ Then $\textup{vec}(VX)=(I_n\otimes V)x.$ Moreover, $\mbox{rank}(VX)\leq k$ and $\mbox{tr}((VX)^*(VX))=\textup{tr}(X^*X)=\|x\|^2=1$. Conversely, if $W\in M_{m,n}$ with $\mbox{rank}(W)\leq k$ and $\textup{tr}(W^*W)=1$, let $V\in M_{m,k}$ with $V^*V=I_k$ whose column space contains that of $W$; then $X=V^*W\in M_{k,n}$ defines an $x=\textup{vec}(X)\in \mathbb{C}^{nk}$ with $\|x\|^2=\textup{tr}(X^*X)=\textup{tr}((VX)^*(VX))=\textup{tr}(W^*W)=1.$ Moreover, \[\textup{vec}(W)^*ZZ^*\textup{vec}(W)=\sum_{s=1}^q|\textup{vec}(B_s)^*\textup{vec}(W)|^2=\sum_{s=1}^q|\textup{tr}(B_s^*W)|^2.\]
For any $v\in \mathbb{C}^q$, note that $\textup{tr}(F(v)^*W)=v^*y$ where $y=[\textup{tr}(B_s^*W)]_{s=1}^q$. Hence, the above expression is $\displaystyle\max_{\substack{v\in \mathbb C^q\\ \|v\|=1}}|\textup{tr}(F(v)^*W)|^2$.

Hence,
\[\begin{array}{rcl}\beta_k^+
&=&\displaystyle\max_{\substack{V\in M_{m,k}\\ V^*V=I_k}}[-\displaystyle\min_{\substack{x\in \mathbb C^{nk}\\ \|x\|=1}}x^*(I_n\otimes V)^*ZZ^*(I_n\otimes V)x]\\
&=&-\displaystyle\min_{\substack{V\in M_{m,k}\\ V^*V=I_k}}\displaystyle\min_{\substack{x\in \mathbb C^{nk}\\ \|x\|=1}}x^*(I_n\otimes V)^*ZZ^*(I_n\otimes V)x\\
&=&-\displaystyle\min_{\substack{W\in M_{m,n}\\ \textup{rank}(W)\leq k\\ \textup{tr}(W^*W)=1}}\textup{vec}(W)^*ZZ^*\textup{vec}(W)\\
&=&-\displaystyle\min_{\substack{W\in M_{m,n}\\ \textup{rank}(W)\leq k\\ \textup{tr}(W^*W)=1}}\displaystyle\max_{\substack{v\in \mathbb C^q\\ \|v\|=1}}|\textup{tr}(F(v)^*W)|^2.
\end{array}\]

Now, consider the expression for $\beta_k^-$. Then 
\[\begin{array}{rcl}\lambda_{max}((I_n\otimes V)^*ZZ^*(I_n\otimes V))&=&\sigma_1((I_n\otimes V)^*Z))^2\\
&=&\displaystyle\max_{\substack{v\in \mathbb C^q\\ \|v\|=1}}\|(I_n\otimes V)^*Zv\|^2\\
&=&\displaystyle\max_{\substack{v\in \mathbb C^q\\ \|v\|=1}}\|(I_n\otimes V)^*\textup{vec}(F(v))\|^2\\
&=&\displaystyle\max_{\substack{v\in \mathbb C^q\\ \|v\|=1}}\|\textup{vec}(V^*F(v))\|^2\\
&=&\displaystyle\max_{\substack{v\in \mathbb C^q\\ \|v\|=1}}\mbox{tr}(V^*F(v)F(v)^*V).\\
\end{array}\]
Hence, 
\[\begin{array}{rcl}
\beta_k^-
&=&\displaystyle\max_{\substack{V\in M_{m,k}\\ V^*V=I_k}}\displaystyle\max_{\substack{v\in \mathbb C^q\\ \|v\|=1}}\mbox{tr}(V^*F(v)F(v)^*V)\\
&=&\displaystyle\max_{\substack{v\in \mathbb C^q\\ \|v\|=1}}\displaystyle\max_{\substack{V\in M_{m,k}\\ V^*V=I_k}}\mbox{tr}(V^*F(v)F(v)^*V)\\

&=&\displaystyle\max_{\substack{v\in \mathbb C^q\\ \|v\|=1}}\sum_{j=1}^k\sigma_j(F(v)F(v)^*)\\\
\end{array}\]
where the last equality follows from \cite[Corollary 4.3.39]{HJ1} and the fact that $F(v)F(v)^*$ is positive semidefinite.
\end{proof}

The formula for $\beta_k^-$ was established in \cite[Theorem 3.2]{mlynik}; however, the above proof is different.

\begin{corollary}\label{phi+_kpos}
Let $\alpha\in \mathbb{R}$ and $B\in M_{m,n}$ with $m\geq 2$ or $n \geq 2$. The map $\Phi_{\alpha,B}^+$ is $k$-positive if and only if $\alpha\geq 0$.
\end{corollary}
\begin{proof}
By Theorems \ref{beta_pm}-\ref{char_beta_pm}, it suffices to show that $\beta_k^+\geq 0$. Assume $m \geq 2$. Let $y\in \mathbb{C}^n$ be unit and let $x\in \mathbb{C}^m$ be unit such that $x\in \{By\}^\perp$. Note that such a vector exists due to $\textup{dim}\{By\}^\perp\geq m-1>0$ by assumption. Let $W_0= xy^*$. Then $\textup{rank}(W_0)=1\leq k$ and $\textup{tr}(W_0^*W_0)=\textup{tr}(yx^*xy^*)=(x^*x)(y^*y)=1.$ Thus, \[\displaystyle\max_{\substack{v\in \mathbb{C}\\ |v|=1}}|\textup{tr}(F(v)^*W_0)|^2=\displaystyle\max_{\substack{v\in \mathbb{C}\\ |v|=1}}|\textup{tr}(\overline{v}B^*xy^*)|^2=\displaystyle\max_{\substack{v\in \mathbb{C}\\ |v|=1}}|\overline{v}y^*B^*x|^2=0.\] This implies $\beta_k^+\geq 0$.

Suppose $n \geq 2$. Let $x \in \mathbb{C}^m$ be unit and let $y \in \mathbb{C}^n$ be unit such that $y \in \{B^* x\}^\perp$. Such a vector exists because $\dim\{B^*x\}^\perp \geq n - 1 > 0$. Similarly, choosing $W_0 = xy^*$ shows that $\beta_k^+ \geq 0$.
\end{proof}

\begin{corollary}\label{phi-_kpos}
Let $\alpha\in \mathbb{R}$ and $B\in M_{m,n}$. The map $\Phi_{\alpha,B}^-$ is $k$-positive if and only if $\alpha\geq \displaystyle\sum_{j=1}^k\sigma_j(B)^2$.
\end{corollary}
\begin{proof}
Since 
\[\beta_k^-=\displaystyle\max_{\substack{v\in \mathbb C\\ |v|=1}}\sum_{j=1}^k\sigma_j(F(v)F(v)^*)=\displaystyle\max_{\substack{v\in \mathbb C\\ |v|=1}}\sum_{j=1}^k\sigma_j((vB)(vB)^*)=\sum_{j=1}^k\sigma_j(BB^*),\]
the claim follows by Theorems \ref{beta_pm}-\ref{char_beta_pm}.
\end{proof}

The proof of the next result is analogous to the proof of Theorem \ref{beta_pm}.

\begin{theorem}\label{gamma_pm}
Let $\Phi^{\pm}_{\alpha,\mathfrak{B}}$ be as in Definition \ref{phi_pm}. Define \[\gamma_k^\pm=\displaystyle\max_{\substack{V\in M_{m,k} \\ V^*V=I_k}}\{\lambda_{max}((I_n\otimes V)^*(\mp ZZ^*)^\tau(I_n\otimes V))\}.\] Then $\Phi^{\pm}_{\alpha,\mathfrak{B}}$ is $k$-copositive if and only if $ \alpha\geq \gamma_k^\pm.$
\end{theorem}

We only consider the analogue of Theorem \ref{char_beta_pm} when $q=1$ due to the difficulty of handling the $q>1$ case.
The next lemma is useful in identifying $\gamma_k^{\pm}$ when $Z=\textup{vec}(B).$

\begin{lemma}\label{evals_v*b}
Let $B\in M_{m,n}$, $V\in M_{m,k}$ with $V^*V=I_k$. Set $C=(I_n\otimes V)^*(\textup{vec}(B)\textup{vec}(B)^*)^\tau (I_n\otimes V)$ and $r=\textup{rank}(V^*B)$. The following hold:
\begin{enumerate}[(i)]
\item If $r=0$, then $C=0$.
\item If $r=1$, then the eigenvalues of $C$ are 
$\sigma_1(V^*B)^2$ and $0$ with multiplicity $nk-1$.
\item If $r\geq 2$, then the eigenvalues of $C$ are $\sigma_i(V^*B)$ for $i\in \{1,\ldots,r\}$, $\pm\sigma_i(V^*B)\sigma_j(V^*B)$ for $i,j\in \{1,\ldots,r\}$ with $i<j$, and $0$ with multiplicty $nk-r^2.$
\end{enumerate}
\end{lemma}
\begin{proof}
Let $B=X\Sigma W^*$ be a singular value decomposition of $B$ and $p=\min\{m,n\}$. Write $X=[x_1\ \cdots\ x_m]\in M_m$ and $W=[w_1\ \cdots\ w_n]\in M_n$. Since $\textup{vec}(B)=\textup{vec}(X\Sigma W^*)=(\overline{W}\otimes X)\textup{vec}(\Sigma ), $ it follows that 
\[\begin{array}{rcl}
C&=&(I_n\otimes V)^*\left(\displaystyle\sum_{i,j=1}^p\sigma_i(B)\sigma_j(B)\overline{w_i}w_j^\top\otimes x_ix_j^* \right)^\tau (I_n\otimes V)\\
&=&(I_n\otimes V)^*\left(\displaystyle\sum_{i,j=1}^p\sigma_i(B)\sigma_j(B)w_jw_i^*\otimes x_ix_j^* \right)(I_n\otimes V)\\
&=&\displaystyle\sum_{i,j=1}^p\sigma_i(B)\sigma_j(B)w_jw_i^*\otimes V^*x_ix_j^*V \\
&=&\displaystyle\sum_{i,j=1}^pw_jw_i^*\otimes g_ig_j^*
\end{array}\]
where $g_i=\sigma_i(B)V^*x_i$ for all $i\in \{1,\ldots,p\}$.

If $r=0$, then $\sigma_i(B)=0$ for all $i$ and so $C=0$. Assume $r\geq 1$ and consider $\widetilde{C}=(W\otimes I_k)^*C(W\otimes I_k)$ which is unitarily similar to $C$. Then 
\[\widetilde{C}=\displaystyle\sum_{i,j=1}^pW^*w_jw_i^*W\otimes g_ig_j^*=\sum_{i,j=1}^pE_{ji}^{(n)}\otimes g_ig_j^*=M\oplus 0_{nk-pk}\]
where $M=\displaystyle\sum_{i,j=1}^pE_{ji}^{(p)}\otimes g_ig_j^*$. At this point, the eigenvalues of $C$ are the eigenvalues of $M$ and $0$ with multiplicity $nk-pk$. We now compute the remaining eigenvalues from $M$. Let $G=[g_1\ \cdots\ g_p]$. Observe that $GG^*=V^*BB^*V.$
Let $G=U\Sigma_GQ^*$ be a singular value decomposition where $U=[u_1\ \cdots\ u_k]\in M_k$ and $Q=[q_1\ \cdots\ q_p]\in M_p$ are unitary. For each $j\in \{1,\ldots,p\}$, \[Gq_j=U\Sigma_GQ^*q_j=U\Sigma_Ge_j=\sigma_j(G)u_j\] and for each $i\in \{1,\ldots,k\}$,
\[G^*u_i=Q\Sigma_G^*U^*u_i=Q\Sigma_G^*e_i=\sigma_i(G)q_i.\] Let $e_1,\ldots,e_p$ be the standard basis vectors of $\mathbb{C}^p$. For any $i,j\in \{1,\ldots,p\}$,
\[\|q_j\otimes u_i\|^2=(q_j\otimes u_i)^*(q_j\otimes u_i)=q_j^*q_j\otimes u_i^*u_i=1\]
and 
\[\begin{array}{rcl}
M(q_j\otimes u_i)&=&\displaystyle\sum_{a,b=1}^pe_b(e_a^*q_j)\otimes g_a(g_b^*u_i)\\
&=&\displaystyle\sum_{a,b=1}^p(g_b^*u_i )e_b\otimes(e_a^*q_j) g_a\\

&=&\displaystyle\sum_{b=1}^p(g_b^*u_i )e_b\otimes\displaystyle\sum_{a=1}^p (e_a^*q_j)g_a\\
&=&G^*u_i\otimes Gq_j\\
&=&\sigma_i(G)q_i\otimes \sigma_j(G)u_j\\

&=&\sigma_i(G)\sigma_j(G) ( q_i\otimes u_j).\\
\end{array}\]
For all $i\in \{1,\ldots,r\}$, $\sigma_i(G)^2=\sigma_i(V^*BW)^2=\sigma_i(V^*B)^2$ is an eigenvalue of $C$. If $r=1$, then $\sigma_1(V^*B)^2$ is the only nonzero eigenvalue of $C$ while $0$ has multiplicity $(nk-pk)+(pk-1)=nk-1$. Assume $r\geq 2$. For $i,j\in \{1,\ldots, r\}$ with $i<j$,
\[\|q_j\otimes u_i\pm q_i\otimes u_j\|^2=(q_j\otimes u_i\pm q_i\otimes u_j)^*(q_j\otimes u_i\pm q_i\otimes u_j)=2\]
and 
\[\begin{array}{rcl}
M(q_j\otimes u_i\pm q_i\otimes u_j)&=&\sigma_i(G)\sigma_j(G)(q_i\otimes u_j)\pm \sigma_j(G)\sigma_i(G)(q_j\otimes u_i)\\
&=&\pm\sigma_i(G)\sigma_j(G)(q_j\otimes u_i\pm q_i\otimes u_j).\\
\end{array}\]
For $i,j\in \{1,\ldots, r\}$ with $i<j$, $\pm \sigma_i(G)\sigma_j(G)=\pm \sigma_i(V^*B)\sigma_j(V^*B)$ is an eigenvalue of $C.$ This case contributes $2\left(\dfrac{r^2-r}{2}\right)=r^2-r$ nonzero eigenvalues. Hence, the total multiplicity of $0$ is $(nk-pk)+(pk-(r^2-r))-r=nk-r^2$.
\end{proof}

\begin{theorem}\label{char_gamma_pm}
Let $\Phi^{\pm}_{\alpha,B}$ be as in Definition \ref{phi_pm} and $\gamma_k^\pm$ be as in Theorem \ref{gamma_pm}. For the expression of $\gamma_k^+,$ assume $m,n\geq 2$.
If $k=1$ or $\textup{rank}(B)\leq 1$, then
\[\gamma_k^+=0;\]
otherwise, \[\gamma_k^+=\sigma_1(B)\sigma_2(B).\] 
Moreover,
\[\gamma_k^-=\sigma_1(B)^2.\]
\end{theorem}
\begin{proof}
Let $V\in M_{m,k}$ with $V^*V=I_k.$ 

We first consider $\gamma_k^+.$ Observe that there exists $V\in M_{m,k}$ with $V^*V=I_k$ such that $\textup{rank}(V^*B)\geq 2$ if and only if $k\geq 2$ and $\textup{rank}(B)\geq 2$. In this case, Lemma \ref{evals_v*b} guarantees that 
\[\begin{array}{rcl}\lambda_{max}((I_n\otimes V)^*(-ZZ^*)^\tau(I_n\otimes V))&=&-\lambda_{min}((I_n\otimes V)^*(ZZ^*)^\tau(I_n\otimes V))\\
&=&-(-\sigma_1(V^*B)\sigma_2(V^*B))\\
&=&\sigma_1(V^*B)\sigma_2(V^*B).
\end{array}\]
We have
\[\sigma_1(V^*B)\leq \sigma_1(V^*)\sigma_1(B)=\sigma_1(B)\]
and since $m,n\geq 2$ and $\textup{rank}(B)\geq 2$
\[\sigma_2(V^*B)\leq\sigma_1(V^*)\sigma_2(B^*)=\sigma_2(B)\]
due to \cite[Theorem 3.3.16]{HJ2} applied to $\begin{bmatrix}V& 0_{n-k}\end{bmatrix}^*B$. By definition of $\gamma_k^+$, it follows that
\[\gamma_k^+\leq \sigma_1(B)\sigma_2(B).\]
To see why equality holds, write $B=X\Sigma W^*$ in a singular value decomposition, and take $V=[x_1\ \cdots\ x_k]$ where $x_1,\ldots,x_k$ are the first $k$ columns of $X$. Now assume $k=1$ or $\textup{rank}(B)\leq 1$. This guarantees that
\[\gamma_k^+=0\]
by Lemma \ref{evals_v*b}. 

For $\gamma_k^-$, Lemma \ref{evals_v*b} implies that
\[\gamma_k^-=\sigma_1(V^*B)^2.\] Analogously, $\sigma_1(V^*B)\leq\sigma_1(B)$ and so
\[\gamma_k^-\leq \sigma_1(B)^2.\] Equality is attained due to a similar reason as in the previous case.
\end{proof}

\begin{corollary}\label{phi+_kcopos}
Let $\alpha\in \mathbb{R}$ and $B\in M_{m,n}$ with $m,n\geq 2$. If $k=1$ or $\textup{rank}(B)\leq 1$, then the map $\Phi_{\alpha,B}^+$ is $k$-copositive if and only if $\alpha\geq 0$. Otherwise, the map $\Phi_{\alpha,B}^+$ is $k$-copositive if and only if $\alpha\geq \sigma_1(B)\sigma_2(B)$.
\end{corollary}

\begin{corollary}\label{phi-_kcopos}
Let $\alpha\in \mathbb{R}$ and $B\in M_{m,n}$. The map $\Phi_{\alpha,B}^-$ is $k$-copositive if and only if $\alpha\geq \sigma_1(B)^2$.
\end{corollary}

The first of the two results below follows from Corollaries \ref{phi+_kpos} and \ref{phi+_kcopos} while the second follows from Corollaries \ref{phi-_kpos} and \ref{phi-_kcopos}.
\begin{corollary}\label{phi+_kppt}
Let $\alpha\in \mathbb{R}$ and $B\in M_{m,n}$ with $m,n\geq 2$. The following are equivalent:
\begin{enumerate}[(i)]
\item $\Phi_{\alpha,B}^+$ is $k$-PPT;
\item $\Phi_{\alpha,B}^+$ is $k$-copositive;
\item $\alpha\geq \begin{cases}0,& \textup{if}\ k=1\ \textup{or}\ \textup{rank}(B)\leq 1\\ 
\sigma_1(B)\sigma_2(B),& \textup{otherwise}.\end{cases}$
\end{enumerate}
\end{corollary}

\begin{corollary}\label{phi-_kppt}
Let $\alpha\in \mathbb{R}$ and $B\in M_{m,n}$. The following are equivalent:
\begin{enumerate}[(i)]
\item $\Phi_{\alpha,B}^-$ is $k$-PPT;
\item $\Phi_{\alpha,B}^-$ is $k$-positive;
\item $\alpha\geq \displaystyle\sum_{j=1}^k\sigma_j(B)^2$.
\end{enumerate}
\end{corollary}

\section{Singular value inequalities}\label{svi}

In \cite{fulautam}, singular value inequalities of the block matrices in $M_2(M_n)^+$ under the map $\Phi^{\pm}_{\alpha,I_n}$ are given.
In this section, we extend these inequalities using the map $\Phi^{\pm}_{\alpha,B}$.

\begin{lemma} \label{ftlthm5}
	For $U,V \in M_{m,2m}$ and $j \in \{1,\ldots,m\}$, 
	\[ \sigma_j(UU^* + VV^*) \geq 2\sigma_j(UV^*) - \frac{1}{2} \textup{tr}((U-V)(U-V)^*). \]
\end{lemma}

\begin{proof}
	Let $U,V \in M_{m,2m}$ and $j \in \{1,\ldots,m\}$. 
	Observe that
	\[ 2\sigma_j(UU^* + VV^*) \geq \sigma_j((U+V)(U+V)^*) \] 
	since $2(UU^* + VV^*) - (U+V)(U+V)^* = (U-V)(U-V)^* \in M_m^+$ and $2(UU^* + VV^*), (U+V)(U+V)^* \in M_m^+$.
	By \cite[Theorem 8.13]{zhang}, 
	\begin{align*}
		& 2\sigma_j(U^*U + V^*V) \\
						&= \sigma_j(2(U^*U + V^*V) - (U-V)^*(U-V) +(U-V)^*(U-V)) \\
						&\leq \sigma_j (2(U^*U + V^*V) - (U-V)^*(U-V)) + \sigma_1((U-V)^*(U-V)).
	\end{align*}
	Now, $(U-V)^*(U-V) \in M_{2m}^+$ and so
	\[ \sigma_1((U-V)^*(U-V)) \leq \text{tr}((U-V)^*(U-V)). \]
	Meanwhile, by \cite[Theorem IX.4.2]{bhatia1},
	\[ 2\sigma_j(UV^*) \leq \sigma_j(U^*U + V^*V). \]
	From these four inequalities, it follows that
	\begin{align*}
		\sigma_j(UU^* + VV^*) &\geq \dfrac{1}{2} \sigma_j((U+V)(U+V)^*) \\
						&= \frac{1}{2} \sigma_j((U+V)^*(U+V)) \\
						&= \frac{1}{2} \sigma_j(2(U^* U + V^* V) - (U-V)^* (U-V)) \\
						&\geq \sigma_j(U^*U + V^*V) - \frac{1}{2} \sigma_1((U-V)^*(U-V)) \\
						&\geq \sigma_j(U^*U + V^*V) - \frac{1}{2} \text{tr}((U-V)^*(U-V)) \\
						&\geq 2\sigma_j(UV^*) - \frac{1}{2} \text{tr}((U-V)^*(U-V)) \\
						&= 2\sigma_j(UV^*) - \frac{1}{2} \text{tr}((U-V)(U-V)^*). 
	\end{align*}
\end{proof}

\begin{theorem} \label{AM_plus}
  Let $\Phi^{+}_{\alpha,B}$ be as in Definition \ref{phi_pm} and $\begin{bmatrix} A & X \\ X^* & C \end{bmatrix} \in M_2(M_n)^+$. 
  If $|\alpha| \geq \dfrac{1}{2}\sigma_1(B)^2$, then for $j \in \{ 1,\ldots,m \}$,
	\begin{align*}
		2\sigma_{j} (\Phi^{+}_{\alpha,B}(X)) \leq 2 (\sigma_j(BXB^*) + |\alpha| |\textup{tr} X|) \leq \sigma_j(\Phi^{+}_{|\alpha|,B}(A + C)).
	\end{align*}
\end{theorem}

\begin{proof}
	Let $j \in \{ 1,\ldots,m \}$.
	Using \cite[Theorem 8.13]{zhang}, we obtain that
	\begin{align*}
		\sigma_j(\Phi^{+}_{\alpha,B}(X)) &= \sigma_j(\alpha(\text{tr} X) I_m + BXB^*) \\
						&\leq \sigma_j(BXB^*) + \sigma_1(\alpha(\text{tr} X)I_m) \\
						&= \sigma_j(BXB^*) + |\alpha||\text{tr} X|
	\end{align*}
	and the first inequality immediately follows.
	
	For the second inequality, choose $\theta \in \mathbb{R}$ such that $\text{tr}(e^{i\theta} X) =|\textup{tr}(X)|$. 
	By the assumption, there exist $E,F \in M_{n,2n}$ such that
	\begin{equation*}	
		\begin{bmatrix} A & e^{i\theta} X \\ e^{-i\theta} X^* & C \end{bmatrix} = \begin{bmatrix} E \\ F \end{bmatrix} \begin{bmatrix} E \\ F \end{bmatrix}^* = \begin{bmatrix} EE^* & EF^* \\ FE^* & FF^* \end{bmatrix}
	\end{equation*}
	and there exist $U,V \in M_{m,2m}$ such that
	\begin{equation} \label{equationThm5.2}
			\begin{bmatrix} BAB^* & e^{i\theta} BXB^* \\ e^{-i\theta} BX^*B^* & BCB^* \end{bmatrix} = \begin{bmatrix} U \\ V \end{bmatrix} \begin{bmatrix} U \\ V \end{bmatrix}^* = \begin{bmatrix} UU^* & UV^* \\ VU^* & VV^* \end{bmatrix}.
	\end{equation}
	Observe that
	\begin{align*}
		\text{tr}((U-V)(U-V)^*) &= \text{tr}( UU^* - UV^* - VU^* + VV^*) \\
						&= \text{tr} ( BAB^* - e^{i\theta} BXB^* - e^{-i\theta} BX^*B^* + BCB^* ) \\
						&= \text{tr} ( B(A - e^{i\theta}X - e^{-i\theta}X^* +V)B^*) \\
						&= \text{tr} ( B( EE^* - EF^* - FE^* + FF^*)B^*) \\
						&= \text{tr} ( B(E-F)(E-F)^* B^*).
	\end{align*}
	Since $\sigma_1(B)^2 I_n - B^*B \in M_n^+$ and $(E-F)(E-F)^* \in M_n^+$, it follows that
	\begin{align*}
		\text{tr}((U-V)(U-V)^*) &= \text{tr} ( B(E-F)(E-F)^* B^*) \\
						&\leq \sigma_1(B)^2 \text{tr} ( (E-F)(E-F)^*) \\
						&= \sigma_1(B)^2 \text{tr}(A - e^{i\theta}X - e^{-i\theta} X^* + C) \\
						&= \sigma_1(B)^2 ( \text{tr}(A+C) - 2 \text{Re}(\text{tr}(e^{i\theta}X)) ) \\
						&= \sigma_1(B)^2 ( \text{tr}(A+C) - 2|\text{tr} X| ).
	\end{align*}
	By Lemma \ref{ftlthm5} applied to \eqref{equationThm5.2} and the above inequality,
	\begin{align*}
		\sigma_j(BAB^* + BCB^*) &= 2 \sigma_j(UU^* + VV^*) \\
						&\geq 2 \sigma_j(UV^*) - \frac{1}{2} \text{tr}((U-V)(U-V)^*) \\
						&\geq 2 \sigma_j (BXB^*) - \sigma_1(B)^2 \left( \frac{1}{2} \text{tr}(A+C) - |\text{tr}( X)| \right).
	\end{align*}
	As a consequence,
	\begin{align*}
		& \sigma_j(\Phi^+_{|\alpha|,B}(A+C)) \\
						&= \sigma_j(|\alpha| \text{tr}(A+C)I_m + BAB^* + BCB^*) \\
						&= \sigma_j(BAB^* + BCB^*) + |\alpha| \text{tr}(A+C) \\
						&\geq 2 \sigma_j (BXB^*) -  \sigma_1(B)^2 \left( \frac{1}{2} \text{tr}(A+C) - |\text{tr} (X)| \right) + |\alpha| \text{tr}(A+C) \\
						&= 2\sigma_j(BXB^*) + 2|\alpha| |\text{tr}X| + (2|\alpha| - \sigma_1(B)^2)  \left( \frac{1}{2} \text{tr}(A+C) - |\text{tr} X| \right).
	\end{align*}
	Recall that the trace map is completely positive (see \cite[Section 2.2.3]{watrous}), and so $\begin{bmatrix} \textup{tr}( A) & \textup{tr}(X )\\ \textup{tr}(X^*) & \textup{tr}( C )\end{bmatrix} \in M_2^+$. Together with the AM-GM inequality, it follows that
	\[ |\textup{tr} X| \leq \sqrt{\text{tr} A \, \textup{tr} C} \leq \frac{1}{2} \textup{tr}(A+C). \]
	Hence, $\dfrac{1}{2} \text{tr}(A+C) - |\text{tr} X| \geq 0 $ and since $2 |\alpha| - \sigma_1(B)^2 \geq 0$ by assumption, 
	\begin{align*}
		\sigma_j(\Phi^+_{|\alpha|,B}(A+C)) \geq 2 \sigma_j(BXB^*) + 2|\alpha| |\text{tr}(X)|.
	\end{align*}\end{proof}

Setting $B = I_n$ in Theorem \ref{AM_plus} and allowing only $\alpha\geq 0$ yields \cite[Theorem 5]{fulautam}. 

Now, we consider generalizations of \cite[Proposition 2]{fulautam} and \cite[Remark 1]{fulautam}. We first prove the following result.

\begin{lemma} \label{lem5.3}
	Let $\alpha \in \mathbb{R}$ and $B \in M_{m,n}$. 
	Define $\Gamma: M_n \to M_m$ by
	\[ \Gamma(Z) = \textup{tr} \left( G Z \right) I_m - (B Z B^*)^\top \text{ for all } Z \in M_n \]
    where $G=\dfrac{2\alpha I_n + B^*B}{m+2}.$ If $\alpha \geq \dfrac{m+1}{2} \sigma_1(B)^2$, then the map $\Gamma$ is 2-positive.
\end{lemma}

\begin{proof}
	First, observe that if a linear map $f : M_m \to M_k$ is 2-positive, then for any $R \in M_{m,n}$, the map $f_R: M_n \to M_k$ given by $f_R(Z) = f(RZR^*)$ is also 2-positive. 

Now, for all $Z \in M_n$,
	\begin{align*}
		\Gamma(Z) &= \text{tr}(GZ) I_m - (B Z B^*)^\top\\
						&= \text{tr}(GZ - B^*BZ + B^*BZ) I_m - (B Z B^*)^\top \\
						&= \text{tr}((G-B^*B)Z) I_m + \text{tr}(BZB^*) I_m - (B Z B^*)^\top\\
						&= \Gamma_1(Z)+\Gamma_2(Z),
	\end{align*}
	where $\Gamma_1(Z) = \text{tr}((G-B^*B)Z) I_m$ and $\Gamma_2(Z) = \text{tr}(BZB^*) I_m - (BZB^*)^\top$.
	To prove that $\Gamma$ is $2$-positive, it suffices to show that $\Gamma_1$ and $\Gamma_2$ are 2-positive.
	
	By assumption, $\alpha \geq \dfrac{m+1}{2} \sigma_1(B)^2$, and so $\alpha I_n - \dfrac{m+1}{2} B^*B \in M_n^+$. Hence,
	\[ G - B^*B = \frac{2\alpha I_n - (m+1)B^*B}{m+2} = \frac{2}{m+2} \left( \alpha I_n - \frac{m+1}{2} B^*B \right) \in M_n^+. \]
	Thus, $\Gamma_1$ is also 2-positive (see \cite[Section 2.2.3]{watrous}). If $m=1$, then $\Gamma_2=0$; if $m>1$, then $\Gamma_2(Z)=(m-1)W(BZB^*)$ where $W(X) = \dfrac{1}{m-1} (\text{tr}(X) I_m - X^\top)$ is one of the Werner-Holevo channels (see \cite[Example 3.36]{watrous}). In either case, $\Gamma_2$ is 2-positive.
\end{proof}

\begin{theorem} \label{AM_minus}
  Let $\Phi^{-}_{\alpha,B}$ be as in Definition \ref{phi_pm} and $\begin{bmatrix} A & X \\ X^* & C \end{bmatrix} \in M_2(M_n)^+$. 
  If $\alpha \geq \dfrac{m+1}{2} \sigma_1(B)^2$, then for $j \in \{ 1,\ldots,m \}$,
	\begin{align*}
		2\sigma_{j} (\Phi^{-}_{\alpha,B}(X)) \leq \sigma_j(\Phi^{-}_{\alpha,B}(A + C)).
	\end{align*}
\end{theorem}

\begin{proof}
	Let $\Gamma$ be as in Lemma \ref{lem5.3}, and let $G = \dfrac{2\alpha I_n + B^*B}{m+2}$.
	Observe that for all $Z \in M_n$,
	\begin{align*}
		\Phi^+_{\frac{1}{2},I_m}(\Gamma(Z)) &= \frac{1}{2} \text{tr}(\Gamma(Z)) I_m + \Gamma(Z) \\
						&= \frac{1}{2} \text{tr} \left( \text{tr}(GZ) I_m - (B Z B^*)^\top \right) I_m + \text{tr}(GZ) I_m - (B Z B^*)^\top \\
						&= \left[ \frac{m}{2} \text{tr}(GZ) - \frac{1}{2} \text{tr}(  B Z B^* ) + \text{tr}(GZ) \right] I_m -  (B Z B^* )^\top \\
						&= \text{tr} \left( \frac{(m+2)GZ - B^*BZ}{2} \right) I_m - (B Z B^* )^\top \\
						&= \text{tr} \left( \frac{(2\alpha I_n + B^*B)Z - B^*BZ}{2} \right) I_m - (B Z B^* )^\top \\
						&= \Phi^-_{\alpha,B}(Z) ^\top.
	\end{align*}
	Since $\alpha \geq \dfrac{m+1}{2} \sigma_1(B)^2$, Lemma \ref{lem5.3} implies that $\Gamma$ is 2-positive. In particular,  
	\[ \begin{bmatrix} \Gamma(A) & \Gamma(X) \\ \Gamma(X)^* & \Gamma(C) \end{bmatrix} \in M_2(M_m)^+. \] 
	By Theorem \ref{AM_plus} applied to the map $\Phi^+_{\frac{1}{2},I_m}$ and the above matrix,
	\[ 2 \sigma_j(\Phi^+_{\frac{1}{2},I_m}(\Gamma(X))) \leq \sigma_j ( \Phi^+_{\frac{1}{2},I_m}(\Gamma(A)+\Gamma(C)))=\sigma_j ( \Phi^+_{\frac{1}{2},I_m}(\Gamma(A+C)))  \]
	or
	\[ 2 \sigma_j(  \Phi^-_{\alpha,B}(X) )=2 \sigma_j(  \Phi^-_{\alpha,B}(X) ^\top ) \leq \sigma_j(  \Phi^-_{\alpha,B}(A+C) ^\top )=\sigma_j( \Phi^-_{\alpha,B}(A+C)  ) \]
	for each $j \in \{1,\ldots,m\}$. 
\end{proof}

The geometric mean of positive definite matrices $A,C \in M_n$ is defined by \[ A \sharp C = A^{\frac{1}{2}} (A^{-\frac{1}{2}} C A^{-\frac{1}{2}})^{\frac{1}{2}} A^{\frac{1}{2}}. \]
For general $A,C \in M_n^+$, the notion of geometric mean is uniquely extended as the following limit in the strong operator topology: $ A \sharp C = \lim\limits_{\epsilon \to 0^+} (A+\epsilon I_n) \sharp (C+\epsilon I_n).$
Properties of the geometric mean such as $A \sharp C = C \sharp A$ \cite[Theorem 4.1.3(i)]{bhatia2} and $\frac{1}{2}(A+C) - A \sharp C \in M_n^+$ \cite[Inequality (4.16)]{bhatia2} will be used to improve the inequalities in Theorems \ref{AM_plus} and \ref{AM_minus}. 

\begin{theorem} \label{GM}
	Let $\Phi^{\pm}_{\alpha,P}$ and $\Phi^{+}_{\beta,Q}$ be as in Definition \ref{phi_pm} such that $P \in M_{m,n}$, $Q \in M_m$, and $m,n \geq 2$.
	Let $\begin{bmatrix} A & X \\ X^* & C \end{bmatrix} \in M_2(M_n)^+$.
	If $\alpha \geq \alpha^\pm$, where \[ \alpha^+ = \sigma_1(P) \sigma_2(P)  \text{ and } \alpha^- = \sigma_1(P)^2 + \sigma_2(P)^2, \] and $\beta \geq \dfrac{1}{2} \sigma_1(Q)^2$, then for $j \in \{1,\ldots,m\}$,
	\[ \sigma_j ( \Phi^{+}_{\beta,Q} ( \Phi^{\pm}_{\alpha,P}(X) ) ) \leq \lambda_j( \Phi^{+}_{\beta,Q} ( \Phi^{\pm}_{\alpha,P} (A) )  \sharp \Phi^{+}_{\beta,Q} ( \Phi^{\pm}_{\alpha,P}(C) ) ). \]
\end{theorem}

\begin{proof}
	Since $\alpha \geq \alpha^\pm$, Corollaries \ref{phi+_kppt}-\ref{phi-_kppt} guarantee that the maps $\Phi^{\pm}_{\alpha,P}$ are 2-PPT.
	Hence, the matrix 
	\[  \begin{bmatrix} \Phi^{\pm}_{\alpha,P}(A) & \Phi^{\pm}_{\alpha,P}(X) \\[0.5em] \Phi^{\pm}_{\alpha,P}(X)^* & \Phi^{\pm}_{\alpha,P}(C) \end{bmatrix} \] is PPT. By \cite[Lemma 1.1]{ando}, 
	\[ \begin{bmatrix} \Phi^{\pm}_{\alpha,P}(C) & \Phi^{\pm}_{\alpha,P}(X) \\[0.5em] \Phi^{\pm}_{\alpha,P}(X)^* & \Phi^{\pm}_{\alpha,P}(A) \end{bmatrix} \in M_2(M_m)^+ \]
	and so by \cite[Lemma 3.1]{ando},
	\[ \begin{bmatrix} \Phi^{\pm}_{\alpha,P}(A) \sharp \Phi^{\pm}_{\alpha,P}(C) & \Phi^{\pm}_{\alpha,P}(X) \\[0.5em] \Phi^{\pm}_{\alpha,P}(X)^* & \Phi^{\pm}_{\alpha,P}(C) \sharp \Phi^{\pm}_{\alpha,P}(A) \end{bmatrix} \in M_2(M_m)^+. \]
	Let $j \in \{1,\ldots,m\}$.
	Since $\beta \geq \dfrac{1}{2} \sigma_1(Q)^2$, Theorem \ref{AM_plus} applied to the map $\Phi^+_{\beta,Q}$ and the above matrix implies 
	\begin{align*}
		2\sigma_j ( \Phi^{+}_{\beta,Q} ( \Phi^{\pm}_{\alpha,P}(X) ) ) &\leq \sigma_j ( \Phi^{+}_{\beta,Q} ( \Phi^{\pm}_{\alpha,P}(A) \sharp \Phi^{\pm}_{\alpha,P}(C) + \Phi^{\pm}_{\alpha,P}(C) \sharp \Phi^{\pm}_{\alpha,P}(A) ) ) \\
						&= 2 \sigma_j ( \Phi^{+}_{\beta,Q} ( \Phi^{\pm}_{\alpha,P}(A) \sharp \Phi^{\pm}_{\alpha,P}(C) ) )
 	\end{align*}
	using the symmetry of the geometric mean.
	Now, the map $\Phi^{+}_{\beta,Q}$ is positive by Corollary \ref{phi+_kpos}.
	Moreover, using \cite[Theorem 4.1.5]{bhatia2}, we have
	\begin{align*}
		\sigma_j ( \Phi^{+}_{\beta,Q} ( \Phi^{\pm}_{\alpha,P}(A) \sharp \Phi^{\pm}_{\alpha,P}(C) ) ) 
						&= \lambda_j ( \Phi^{+}_{\beta,Q} ( \Phi^{\pm}_{\alpha,P}(A) \sharp \Phi^{\pm}_{\alpha,P}(C) ) ) \\
						&\leq \lambda_j( \Phi^{+}_{\beta,Q} ( \Phi^{\pm}_{\alpha,P} (A) )  \sharp \Phi^{+}_{\beta,Q} ( \Phi^{\pm}_{\alpha,P}(C) ) ).
	\end{align*}
	This proves the desired inequality.
\end{proof}

Theorem \ref{GM} generalizes the inequalities in \cite[Proposition 3]{fulautam}. In particular, we have the following result.

\begin{corollary}
	Let $\Phi^{\pm}_{\gamma,P}$ be as in Definition \ref{phi_pm} and $\begin{bmatrix} A & X \\ X^* & C \end{bmatrix} \in M_2(M_n)^+$.
	Suppose that $P \in M_{m,n}$ such that $P^*P = I_n$ and $n \geq 2$.
	\begin{enumerate}[(i)]
		\item If $\gamma\geq\frac{m+3}{2}$, then for $j \in \{ 1,\ldots,m \}$,
	\[ \sigma_j(\Phi^{+}_{\gamma,P}(X)) \leq \lambda_j(\Phi^{+}_{\gamma,P}(A) \sharp \Phi^{+}_{\gamma,P}(C)).\]
	
	\item If $\gamma\geq m+\frac{3}{2}$, then for $j \in \{ 1,\ldots,m \}$,
	\[ \sigma_j(\Phi^{-}_{\gamma,P}(X)) \leq \lambda_j(\Phi^{-}_{\gamma,P}(A) \sharp \Phi^{-}_{\gamma,P}(C)).\]
	\end{enumerate}
\end{corollary}

\begin{proof}
	Since $P^*P = I_n$, observe that for any $Z \in M_n$ 
	\begin{align*}
		(\Phi^{+}_{\beta,I_m} ( \Phi^{\pm}_{\alpha,P}(Z)) &= \Phi^{+}_{\beta,I_m} ( \alpha \text{tr} (Z) I_m \pm PZP^* ) \\
						&= \alpha \text{tr} (Z)\Phi^{+}_{\beta,I_m}(I_m) \pm \Phi^{+}_{\beta,I_m}(PZP^*) \\
						&= \alpha \text{tr} (Z) [\beta \text{tr} (I_m) I_m + I_m] \pm [\beta \text{tr}( PZP^*) I_m + PZP^*] \\
						&= (\alpha \beta m + \alpha \pm \beta) \text{tr} (Z) I_m \pm P Z P^* \\
						&= \Phi^{\pm}_{\gamma,P}(Z)
	\end{align*}
    where $\gamma=\alpha \beta m + \alpha \pm \beta.$ Note that $m \geq n \geq 2$ and $\sigma_1(P) = \sigma_2(P) = 1$. 
	If we set $\alpha^+ = 1$ and $\alpha^- = 2$ and consider $\alpha \geq \alpha^\pm$ and $\beta \geq \frac{1}{2}$, then these give the two assumptions $\gamma\geq \frac{m+3}{2}$ and $\gamma\geq m+\frac{3}{2}$ corresponding to $\Phi^\pm_{\gamma, P}.$ Now Theorem \ref{GM} ensures that
	\[ \sigma_j ( \Phi^{+}_{\beta,I_m} ( \Phi^{\pm}_{\alpha,P}(X) ) ) \leq \lambda_j( \Phi^{+}_{\beta,I_m} ( \Phi^{\pm}_{\alpha,P} (A) )  \sharp \Phi^{+}_{\beta,I_m} ( \Phi^{\pm}_{\alpha,P}(C) ) )\]
	for $j \in \{1,\ldots,m\}$. Rewriting the expression in terms of $\gamma$ implies the inequalities.
\end{proof}







\bigskip
{\bf Acknowledgments.} 
The research of the first author was supported by University of the Philippines Diliman, Natural Sciences Research Institute Project MAT-25-1-01.


\end{document}